\documentclass[12pt]{article}

\usepackage{a4wide}
\usepackage{amssymb,amsmath,array,amsthm, amsfonts, amscd}
\usepackage{colortbl}
\usepackage{algpseudocode}
\usepackage{algorithm}
\usepackage{tikz}
\usepackage{graphicx}
\usepackage{comment}
\allowdisplaybreaks[1]

\input xy
\xyoption{matrix}\xyoption{arrow}
\def\edge{\ar@{-}}
\def\dedge{\ar@{.}}

\newtheorem{theorem}{Theorem}[section]
\newtheorem{corollary}[theorem]{Corollary}
\newtheorem{lemma}[theorem]{Lemma}
\newtheorem{proposition}[theorem]{Proposition}

\theoremstyle{definition}
\newtheorem{definition}[theorem]{Definition}
\newtheorem{example}[theorem]{Example}
\newtheorem{remark}[theorem]{Remark}

\def\cc{{\mathcal C}}

\def\ch{{\mathcal H}}

\def\cm{{\mathcal M}}

\def\co{{\mathcal O}}

\newcommand\mcr{{\mathcal R}}

\def\gkdim{{\rm GKdim}}
\def\st{{\rm st}}

\def\minusqdot{(-q)^\bullet}
\def\qdot{q^\bullet} 
\def\qhat{\widehat{q}}

\def\oq{\mathcal{O}_{q}}

\def\oqmnk{\oq(M_{mn}(\k))}
\def\oqnnk{\oq(M_{nn}(\k))}
\def\oqmnkg{\oq(M_{mn}(\k))_{\gamma}}

\def\k{K}

\title{Generalised quantum determinantal\\
 rings are 
maximal orders}
\author{T H Lenagan and L Rigal}
\date{}

\begin{document}
\maketitle
\begin{abstract} Generalised quantum determinantal rings are the analogue in quantum matrices of Schubert varieties. Maximal orders are the noncommutative version of integrally closed rings. 
In this paper, we show that generalised quantum determinantal rings 
are maximal orders. The cornerstone of the proof is a description of generalised quantum
determinantal rings, up to a localisation, as skew polynomial extensions.
\end{abstract}




\vskip .5cm
\noindent


\section{Introduction} 

Let $\k$ be a field, let $m,n$ be positive integers and let $q$ be a nonzero element of $\k$.
The algebra of quantum matrices over $\k$, denoted by $\oqmnk$, is a quantum deformation of the 
coordinate ring of the variety of $m\times n$ matrices over $\k$. The set of quantum minors $\Pi$ 
in $\oqmnk$ 
carries a natural partial order with respect to which the standard monomials form a basis over $\k$: more precisely, $\oqmnk$ is a quantum graded algebra with a straightening law on the poset of quantum minors equipped with the standard partial order. (Precise definitions are given later.) 

Given a quantum minor $\gamma$ in $\oqmnk$, one can define a factor ring of $\oqmnk$, denoted 
by $\oqmnkg$, and known as 
the generalised quantum determinantal ring/factor  determined by $\gamma$. The generalised quantum 
determinantal factors of $\oqmnk$ are the analogues for $\oqmnk$ of  the quantum Schubert varieties in the 
grassmannian studied in \cite{lr-qsch}. The term generalised quantum determinantal ring is used because special instances of 
$\gamma$ determine the quantum determinantal factors where all quantum minors of a given size are set to be zero. Generalised quantum determinantal rings were shown to be integral domains  in \cite[Proposition 4.3]{lr-qsch}, but the question as to whether or not they are maximal orders was left open, see \cite[Remark 4.6]{lr-qsch}. 
In this work we show that they are indeed maximal orders. Maximal orders are the noncommutative 
analogues of normal varieties, or integrally closed rings. 

Our motivation, here, comes from noncommutative algebraic geometry. The algebras that we study are 
noncommutative analogues of coordinate rings of natural varieties arising from Lie theory and we want to 
study them as such. This was already the point of view in the works 
\cite{lr-max-det}, \cite{lr-qasl} and \cite{lr-qsch}, where properties of geometric nature of related 
algebras, expressible either in ring theoretic language (integrity, normality), or homologically 
(AS-Cohen-Macaulay, AS-Gorenstein properties) were studied.



\section{Basic definitions}

Let $\k$ be a field, and let $q$ a nonzero element of $\k$. 
The algebra of $m\times n$ quantum matrices over $\k$, denoted by $\oqmnk$, is 
the algebra generated over $\k$ by 
$mn$ indeterminates 
$x_{ij}$, with $1 \le i \le m$ and $1 \le j \le n$,  which commute with the elements of 
$\k$ and are subject to the relations:
\[
\begin{array}{ll}  
x_{ij}x_{il}=qx_{il}x_{ij},&\mbox{ for }1\le i \le m,\mbox{ and }1\le j<l\le
n\: ;\\ 
x_{ij}x_{kj}=qx_{kj}x_{ij}, & \mbox{ for }1\le i<k \le m, \mbox{ and }
1\le j \le n \: ; \\ 
x_{ij}x_{kl}=x_{kl}x_{ij}, & \mbox{ for } 1\le k<i \le m,
\mbox{ and } 1\le j<l \le n \: ; \\
x_{ij}x_{kl}-x_{kl}x_{ij}=(q-q^{-1})x_{il}x_{kj}, & \mbox{ for } 1\le i<k \le
m, \mbox{ and } 1\le j<l \le n.
\end{array}
\]
It is well-known that $\oqmnk$ is an iterated skew polynomial extension of $\k$ 
with the $x_{ij}$ added 
in lexicographic order. An immediate consequence is that $\oqmnk$ is a noetherian domain. 

When $m=n$, the {\em quantum determinant} $D_q$ is defined by;
\[
D_q:= \sum\,(-q)^{l(\sigma)}x_{1\sigma(1)}\dots x_{n\sigma(n)},
\]
where the sum is over all permutations $\sigma$ of $\{1,\dots,n\}$. 

The quantum determinant is a central element in the algebra 
$\oqnnk$. \\

Let $I$ and $J$ be $t$-element subsets of $\{1,\dots,m\}$ and $\{1,\dots,n\}$, respectively.
It is clear from the definitions that the subalgebra of $\oqmnk$ generated by those $x_{ij}$ with
$i\in I$ and $j\in J$ is isomorphic in the obvious way to ${\mathcal O}_q(M_{tt}(K))$. 
Then the {\em quantum minor} $[I\mid J]$ is defined to be the quantum determinant of this 
subalgebra. 
(Note that $x_{ij}=[i\mid j]$ and $[\emptyset\mid\emptyset]$ is taken to be $1$.)
It is immediate that $x_{ij}[I\mid J]=[I\mid J]x_{ij}$ for $i\in I$ and $j\in J$, but quantum minors do not commute with other variables. Nevertheless, several useful commutation relations have been developed, and 
we will use some of them in this article.

The set of all quantum minors  is denoted by $\Pi$.  The set
$\Pi$ is equipped with the partial order $\le_\st$ defined in
\cite[Section 3.5]{lr-qasl}. 
Namely, if $[I\mid J]$ and $[K\mid L]$ are quantum minors  with $I=\{i_1< \dots <i_u\}, J=\{j_1< \dots <j_u\},  
K=\{k_1< \dots <k_v\}$ and $L=\{l_1< \dots <l_v\}$ 
then 
\[ [I\mid J] \le_\st [K\mid L]
\Longleftrightarrow \left\{
\begin{array}{l}
u \ge v, \cr 
i_s \le k_s \quad\mbox{for}\quad 1 \le s \le v , \cr
j_s \le l_s \quad\mbox{for}\quad 1 \le s \le v .
\end{array}
\right.
\]~\\

The algebra of quantum matrices, equipped with the partial order $\le_\st$ defined on the set of quantum minors $\Pi$,  
is a quantum graded algebra with a straightening law (abbreviated QGASL), as defined in \cite{lr-qasl}, see \cite[Theorem 3.5.3]{lr-qasl}.

\begin{definition}\label{definition-qgasl}
Let $\gamma\in\Pi$ and set $\Pi_\gamma:=\{\alpha\in\Pi\mid\alpha\not\geq_{\st}\gamma\}$. 
Set $I_\gamma$ to be the ideal generated by $\Pi_\gamma$. 
The {\em 
generalised quantum determinantal ring} $\oqmnkg$ associated to $\gamma$ is the factor algebra 
$\oqmnk/I_\gamma$. 
(We let $p \, : \, \oqmnk \longrightarrow \oqmnkg$ be the canonical projection.)
\end{definition} 

The terminology we use arises in the following way. Let $\gamma = [1,\dots,t-1\mid 1,\dots,t-1]$. 
Then $\Pi_\gamma$ consists of the $s\times s$ quantum minors with $s\geq t$, and $\oqmnkg$ 
is the factor ring obtained by setting all of the $t\times t$ quantum minors to be zero: such algebras are known as quantum determinantal rings, see, for example, \cite{lr-max-det}

\begin{proposition}
The generalised quantum determinantal ring $\oqmnkg$ is a QGASL on the natural projection of
$\Pi\backslash\Pi_\gamma$ from $\oqmnk$ to  $\oqmnkg$.
\end{proposition}

\begin{proof} 
This follows immediately from \cite[Theorem 3.5.3 and Corollary 1.2.6]{lr-qasl}.
\end{proof}

Recall that an element $u$ of a ring $R$ is a {\em normal} element if $uR=Ru$ and is {\em regular} if 
it is a nonzerodivisor. 

\begin{corollary} \label{corollary-minimal-element}
The image $\overline{\gamma}$ of $\gamma$ in $\oqmnkg$ is the unique minimal element of 
$p(\Pi\backslash\Pi_\gamma)$. 
Further, for each $\tau\geq_\st\gamma$, there exists $c_\tau\in\k$, nonzero, such that $\overline{\gamma}\,\overline{\tau}=c_\tau\overline{\tau}\,\overline{\gamma}$, where $\overline{\tau}=p(\tau)$. 
Consequently, $\overline{\gamma}$ 
is a regular normal element of the generalised quantum determinantal ring $\oqmnkg$.
\end{corollary}

\begin{proof} See the proof of \cite[Lemma 1.2.1]{lr-qasl}.
\end{proof} 



\section{Relations for a subalgebra of quantum matrices}\label{section-relations} 


Let $\gamma = [A|B]=[a_1,\dots,a_t\mid b_1,\dots, b_t]$ be a quantum minor in $\oqmnk$, and let 
$c_1 < c_2 < \dots < c_{n-t}$ be the column indices of $\oqmnk$ that do not
occur in $\gamma$ and $r_1 < r_2 < \dots < r_{m-t}$ be the row indices of $\oqmnk$ that do not
occur in $\gamma$. We will use this notation throughout the paper.

Denote by $S$ the $t\times t$ quantum matrix subalgebra of $\oqmnk$ generated by the $x_{a_ib_j}$.
Our strategy to show that the 
generalised determinantal algebra determined by $\gamma$ is a maximal order will be to show that 
it is related via localisation to an algebra $T$ which is an iterated Ore extension. The algebra $T$ is a 
subalgebra of quantum matrices 
generated by $S$ and a family of quantum minors
(explicit generators are given below). In order to show that $T$ 
is an iterated Ore extension, we need to do two things. First, we need to develop suitable
commutation relations between the generators; this is done in this section. Secondly, we need to show that 
the generators are independent enough to give a presentation as an iterated Ore extension.

Let $\cm$ be the set of $t\times t$  quantum
minors that are $\geq_\st \gamma$, and which differ from
$\gamma$ in precisely one entry.
Let $T$ be the subalgebra of $\oqmnk$ generated over $S$ 
by the quantum minors in $\cm$.  
To study $T$, we need a notation for the quantum minors in $\cm$.

First, note that each such minor either has the same row set or column set as
$\gamma$ and differs from $\gamma$ in the column set or row set, respectively,
by exactly one element. Let $\mcr$ be the set of such quantum 
minors with the same row set
as $\gamma$ and $\cc$ be the set of such quantum 
minors with the same column set as
$\gamma$.

Note that a quantum minor $[A\mid B\sqcup\{c_{(n-t+1)-i}\}\backslash\{b_j\}]$ is in $\cm$ precisely 
when $b_j<c_{(n-t+1)-i}$. Similarly, the quantum minor $[A\sqcup\{r_{(m-t+1)-i}\}\backslash\{a_j\}\mid B]$ 
is in $\cc$ 
precisely when $a_j<r_{(m-t+1)-i}$.

For each $i,j$ such that $b_j<c_{(n-t+1)-i}$, set 
\[
m_{ij}:=[A\mid B\sqcup\{c_{(n-t+1)-i}\}\backslash\{b_j\}],
\]
and, similarly, set 
\[
n_{ij}:=[A\sqcup\{r_{(m-t+1)-i}\}\backslash\{a_j\}\mid B]
\]
 whenever 
 $a_j<r_{(m-t+1)-i}$.
 
 Then, $\mcr=\{m_{ij}\mid b_j<c_{(n-t+1)-i}\}$ and $\cc=\{n_{ij}\mid a_j<r_{(m-t+1)-i}\}$. 

We now order $\cm=\mcr\sqcup\cc$ in the following way. The members  of $\mcr$ come before 
the members of $\cc$, while $m_{ij}\leq m_{kl}$ if and only if $(i,j)\leq (k,l)$ in lexicographic order and, 
similarly, $n_{ij}\leq n_{kl}$ if and only if $(i,j)\leq (k,l)$ in lexicographic order.

The plan is to build up $T$ from $S$ by introducing the members of $\cm$ in this order. \\

At this point, let's look at a specific example.

\begin{example}
Consider the quantum minor $\gamma=[13\mid 12]$ in $\oq(M_{3,3}(K))$. Then 
$S=\k[x_{11},x_{12},x_{31},x_{32}]$, a quantum matrix subalgebra of $\oq(M_{3,3}(K))$, while $\mcr=\{m_{11},m_{12}\}=\{[13\mid 23],[13\mid 13]\}$, $\cc=
\{n_{11}\}=\{[23\mid12]\}$ and 
$\cm= \{m_{11}<m_{12}<n_{11}\}=\{[13\mid 23]<[13\mid13]<[23\mid 12]\}$.
Then, 
\[
T=S[m_{11},m_{12},n_{11}].
\]
What we are aiming to do amounts to showing, in this example, that $T$ is (isomorphic to) a three step iterated Ore extension of $S$, with the ``variables'' $m_{11},m_{12},n_{11}$ added in this order. It then 
follows that $T$ is a seven step iterated Ore extension of $\k$, and that $\gkdim(T)=7$.
\end{example} 

For each $m_{kl}$ that is defined let $R(m_{kl})$ be the subalgebra of $T$ generated by $S$ and the 
$m_{ij}$ that are less than or equal to $m_{kl}$ in the order defined above, and for each 
$n_{kl}$ that is defined let $R(n_{kl})$ be the subalgebra of $T$ generated by $S$, all of the $m_{ij}$ and the 
$n_{ij}$ that are less than or equal to $n_{kl}$ in the order defined above. Let $m_{kl}^-$ be the $m_{ij}$
that immediately precedes $m_{kl}$ in the above order and let 
$n_{kl}^-$ be the $n_{ij}$
that immediately precedes $n_{kl}$ in the above order. Then, $R(m_{kl})$ is generated over 
$R(m_{kl}^-)$ by $m_{kl}$ and $R(n_{kl})$ is generated over 
$R(n_{kl}^-)$ by $n_{kl}$. (If $m_{kl}^-$ does not exist set $R(m_{kl}^-):=S$, and,  similarly, if $n_{kl}^-$ does not exist then $R(n_{kl}^-)$ is generated over 
$S$ by all of the $m_{ij}$.)

We need to know suitable commutation relations between the members of $\cm$ and between 
members of $\cm$ 
and the $x_{ij}$ in the quantum matrix subalgebra $S$. 

First, we check how the $m_{ij},n_{ij}$ commute with the generators of $S$. 

\begin{lemma}
Let $x_{a_k b_l}$ be a generator for $S$. Then 

(i)
$x_{a_k b_l}$ commutes with $m_{ij}$ when  $l\neq j$, while 
\[
x_{a_k b_j}m_{ij} -qm_{ij}x_{a_k b_j} 
=
\qhat \sum_{s<j}\, \minusqdot m_{is}x_{a_k b_s}\,,
\]
where $\qhat = q-q^{-1}$ and $\minusqdot$ is an unspecified positive integer power of $-q$.

(ii) $x_{a_k b_l}$ commutes with $n_{ij}$ when  $k\neq j$, while 
\[
x_{a_j b_l}n_{ij} -qn_{ij}x_{a_j b_l}
=
\qhat \sum_{s<j}\, \minusqdot n_{is}x_{a_s b_l}\,,
\]
where $\qhat = q-q^{-1}$ and $\minusqdot$ is an unspecified positive integer power of $-q$.
\end{lemma}

\begin{proof} We prove (i), the proof of (ii) is similar. 
Recall that $m_{ij}=[A\mid B\sqcup\{c_{(n-t+1)-i}\}\backslash\{b_j\}]$ is the 
quantum determinant of a quantum matrix subalgebra that contains the element 
$x_{a_k b_l}$ whenever $l\neq j$, and so commutes with such elements. 

When $l=j$, we use \cite[Lemma 4.5.1(2), first equation]{pw}
with $r=a_k,\; 
s=b_j,\; A(i,s)=[A\mid B\sqcup\{c_{(n-t+1)-i}\}\backslash\{b_j\}]$. (Note that we must interchange $q$ and $q^{-1}$ when using results from \cite{pw}.)
We see that 
\begin{eqnarray*}
\lefteqn{x_{a_k b_j}m_{ij} -qm_{ij}x_{a_k b_j} }\\
&=&
x_{a_k b_j} [A\mid B\sqcup\{c_{(n-t+1)-i}\}\backslash\{b_j\}]        
-q[A\mid B\sqcup\{c_{(n-t+1)-i}\}\backslash\{b_j\}]x_{a_k b_j}\\
&=& 
\qhat\sum_{s<j}\,\minusqdot 
[A\mid (B\sqcup\{c_{(n-t+1)-i}\}\backslash\{b_j\}) \sqcup \{b_j\}\backslash \{b_s\}]x_{a_k b_s}\\
&=&
\qhat\sum_{s<j}\,\minusqdot[A\mid  B\sqcup\{c_{(n-t+1)-i}\}
\backslash\{b_s\}]x_{a_k b_s}\\
&=&
\qhat\sum_{s<j}\,\minusqdot m_{is}x_{a_k b_s}\,.
\end{eqnarray*}
Note that these $m_{is}$ exist, as $b_s<b_j<c_{(n-t+1)-i}$. 
\end{proof} 

It is important to note that this result shows that $x_{a_k b_j}m_{ij} -qm_{ij}x_{a_k b_j}\in R(m_{ij}^-)$, and similarly for $n_{ij}$.\\

Next,  we need to know commutation relations for $\cm$.

\begin{proposition}
(i) The $m_{ij}$ obey the rules for quantum matrix variables, with parameter $q^{-1}$, as do the $n_{ij}$.\\
(ii) Given $m_{ij}\in\mcr$ and $n_{kl}\in\cc$,
$
m_{ij}n_{kl}=n_{kl}m_{ij}.
$
\end{proposition}

\begin{proof}
There are several relations to check. We present the proof of the following claim. (The proofs for all other cases are similar, but easier.) 
{\em Suppose that $i<k$ and $j<l$ and that $m_{ij}$ and $m_{kl}$ are defined. 
Then $m_{il}$ and $m_{kj}$ are defined and 
\[
m_{ij}m_{kl}-m_{kl}m_{ij} = (q^{-1}-q)m_{il}m_{kj}\,.
\]
}

{\em Proof of claim}. 
We know that $b_j<b_l$, and $b_l<c_{(n-t+1)-k}$ as $m_{kl}$ is defined. Thus, 
$b_j<c_{(n-t+1)-k} $ and so $m_{kj}=[A\mid B\sqcup \{c_{(n-t+1)-k}\}\backslash \{b_j\}]$ is defined. 
Also, $b_l<c_{(n-t+1)-k}$, as 
$m_{kl}$ is defined, 
and $ c_{(n-t+1)-k}<c_{(n-t+1)-i}$ as $i<k$. As a consequence,  $b_l<c_{(n-t+1)-i}$
and so
$m_{il}=[A\mid B\sqcup \{c_{(n-t+1)-i}\}\backslash \{b_l\}]$ is defined. 

Set $A'=A\backslash\{a_1,a_2\}$ and 
$B':= B\backslash\{b_j,b_l\}$. 
Then $m_{ij}=[A'\sqcup\{a_1,a_2\}\mid B'\sqcup\{c_{(n-t+1)-i},b_l\}]$
and 
$m_{kl}=[A'\sqcup\{a_1,a_2\}\mid B'\sqcup\{c_{(n-t+1)-k},b_j\}]$. 

We need a commutation rule for the quantum minors $[a_1a_2\mid b_l c_{(n-t+1)-i}]$ 
and $[a_1a_2\mid b_j c_{(n-t+1)-k}]$, where $b_j<b_l<c_{(n-t+1)-k}<c_{(n-t+1)-i}$.
It is easy to verify that, in $\oq(M_{2,4}(K))$, we have 
$[12|24][12|13]-[12|13][12|24] = (q^{-1} -q)[12|14][12|23]$. 
It follows that, in $\oqmnk$,  
\[
\begin{array}{l}
[a_1 a_2|b_l c_{(n-t+1)-i}][a_1 a_2|b_j c_{(n-t+1)-k}] 
- [a_1 a_2|b_j c_{(n-t+1)-k}][a_1 a_2|b_l c_{(n-t+1)-i}] \cr
\qquad\qquad\qquad\qquad\qquad\qquad\qquad\qquad
=(q^{-1} -q)[a_1 a_2|b_j c_{(n-t+1)-i}][a_1 a_2|b_l c_{(n-t+1)-k}].
\end{array}
\]
Using the quantum Muir's Law of extensible minors, see 
\cite[Proposition 1.3]{lr-qsch} for example, to re-introduce $A'$ and $B'$ we obtain
\[
m_{ij}m_{kl}-m_{kl}m_{ij} = (q^{-1}-q)m_{il}m_{kj}\,,
\]
as required. 
\end{proof}



\section{Torus actions induced from $\oqmnk$}\label{section-torus-actions}

The algebra $T$ is a subalgebra of $\oqmnk$. Recall that there is an action of the torus 
$\ch=(\k^*)^{m+n}$ on $\oqmnk$ defined on the generators of $\oqmnk$ in the following way: if 
$h=(\alpha_1,\dots,\alpha_m;\beta_1\dots,\beta_n)$ then $h\cdot x_{ij}:= \alpha_i\beta_jx_{ij}$. The generators $x_{ij},m_{ij},n_{ij}$ are all eigenvectors for the action of $\ch$; and so 
it is easy to 
check that $\ch$ restricts to automorphisms of $T$ and the various subalgebras that we are using to build up $T$ as a purported Ore extension. Our aim is to show that the commutation relations developed in Section~\ref{section-relations} can be rephrased by using suitable choices of elements $h\in\ch$.

\begin{definition} \label{definition-hvalues} 
Set $h_{m_{kl}}:= (\alpha_1,\dots,\alpha_m;\beta_1,\dots,\beta_n)$ 
where (i) $\alpha_s=1$ when $s\in A$, and $\alpha_s=q^{-1}$ when $s\not\in A$, and (ii) 
$\beta_{c_{(n-t+1)-k}} =q^{-2}$, $\beta_{b_l}=q$, and $\beta_s=1$ for $s\in B\backslash\{b_l\}$, with 
$\beta_s=q^{-1}$ for $s\not\in B\sqcup \{c_{(n-t+1)-k}\}$.

Also, set $h_{n_{kl}}:= (\alpha_1,\dots,\alpha_m;\beta_1,\dots,\beta_n)$ where (i) $\alpha_{r_{(m-t+1)-k}}
=q^{-2}, \alpha_l=q$, and $\alpha_s=1$ for $s\in A\backslash\{a_l\}$, while $\alpha_s=q^{-1}$ for 
$s\not\in A\sqcup\{r_{(m-t+1)-k}\}$, and (ii) $\beta_s=1$ for $s\in B$ and $\beta_s=q^{-1}$ for $s\not\in B$.
\end{definition}

Let's check the action of the  $h$ that we have just defined on relevant generators of $T$.

\begin{lemma}\label{lemma-h-actions} 
The following hold. \\
(1) For $i\in A$, $j\in B$, then 
$h_{m_{kl}}(x_{ib_{l}})=qx_{ib_{l}}$ and $h_{m_{kl}}(x_{ij})=x_{ij}$ when $j\neq b_l$. \\
(2) For $i\in A$, $j\in B$, then
$h_{n_{kl}}(x_{a_{l}j})=qx_{a_{l}j}$ and $h_{n_{kl}}(x_{ij})=x_{ij}$ when $i\neq a_l$.  \\
(3) Suppose that $(i,j)< (k,l)$ in lexicographic order, then: \\
(a) $h_{m_{kl}}(m_{ij}) = m_{ij}$ when $i\neq k$ and $j\neq l$; \\
(b) $h_{m_{kl}}(m_{ij}) = q^{-1}m_{ij}$ when $i=k$ and $j<l$ or $i<k$ and $j=l$. \\
(4) Suppose that $(i,j) < (k,l)$ in lexicographic order, then: \\
(a) $h_{n_{kl}}(n_{ij}) = n_{ij}$ when $i\neq k$ and $j\neq l$; \\
(b) $h_{n_{kl}}(n_{ij}) = q^{-1}n_{ij}$ when $i=k$ and $j<l$ or $i<k$ and $j=l$. \\
(5) $h_{n_{kl}}(m_{ij}) = m_{ij}$, for all $n_{kl}$ and $m_{ij}$.
\end{lemma} 

\begin{proof} We prove (1) and (3a). The proofs of all other claims are similar to one of these two.\\  
(1) $h_{m_{kl}}(x_{ij}) =\alpha_i\beta_j x_{ij}$. Now $i\in A$, so $\alpha_i=1$ and 
$h_{m_{kl}}(x_{ij}) =\beta_j x_{ij}$, 
which is equal to $qx_{ij}$ when $j=b_l$ and equal to $1.x_{ij}$ otherwise. \\
(3a) Suppose that $i\neq k$ and $j\neq l$. Now, 
$m_{ij}:=[A\mid B\sqcup \{c_{(n-t+1)-i}\}\backslash\{b_j\}]$ with $b_j<c_{(n-t+1)-i}$, 
and $m_{kl}:=[A\mid B\sqcup\{c_{(n-t+1)-k}\}\backslash\{b_l\}]$ with 
$b_l<c_{(n-t+1)-k}$.

Let $h=h_{m_{kl}}=(\alpha_1,\dots,\alpha_m;\beta_1\dots,\beta_n)$. Then 
\begin{eqnarray*}
\lefteqn{h_{m_{kl}}(m_{ij})}\\
&=&
h_{m_{kl}}( [A\mid B\sqcup \{c_{(n-t+1)-i}\}\backslash\{b_j\}])\\
&=&
\alpha_{a_1}\dots\alpha_{a_t}\beta_{b_1}\dots \widehat{\beta_{b_j}}\dots \beta_{b_t}\beta_{c_{(n-t+1)-i}}
[A\mid B\sqcup \{c_{(n-t+1)-i}\}\backslash\{b_j\}]\\
& = &
\alpha_{a_1}\dots\alpha_{a_t}\beta_{b_1}\dots \widehat{\beta_{b_j}}\dots \beta_{b_t}\beta_{c_{(n-t+1)-i}}
m_{ij} 
\end{eqnarray*}

Hence, we need to evaluate $\lambda:= \alpha_{a_1}\dots\alpha_{a_t}\beta_{b_1}
\dots \widehat{\beta_{b_j}}\dots \beta_{b_t}\beta_{c_{(n-t+1)-i}}$. From the definition of $h_{m_{kl}}$ we see that 
each $\alpha_{a_i} =1$. Also, for $s\in B\backslash\{b_l\}$ we know that $\beta_s=1$.
Therefore, $\lambda =\beta_{b_l}\beta_{c_{(n-t+1)-i}}$. 
We know that $\beta_{b_l}=q$, so it remains to evaluate $\beta_{c_{(n-t+1)-i}}$. 
As $i\neq k$ it follows that $c_{(n-t+1)-i}\neq c_{(n-t+1)-k}$ so that $c_{(n-t+1)-i}\not\in B\sqcup \{c_{(n-t+1)-k}\}$ and so $\beta_{c_{(n-t+1)-i}}=q^{-1}$. Hence, $\lambda=\beta_{b_l}\beta_{c_{(n-t+1)-i}}=qq^{-1}=1$.
\end{proof} 

 The previous lemma, together with the results obtained in Section~\ref{section-relations}
 are sufficient to establish  the following result.
 
 \begin{proposition} 
 (i) For $x_{ij}\in S$ and for any $m_{kl}$ that is defined, 
 \[
 m_{kl}x_{ij} - h_{m_{kl}}^{-1}(x_{ij})m_{kl}\in R(m_{kl}^-)
 \]
 
 (ii) For $x_{ij}\in S$ and for any $n_{kl}$ that is defined, 
 \[
 n_{kl}x_{ij} - h_{n_{kl}}^{-1}(x_{ij})n_{kl}\in R(n_{kl}^-)
 \]
 
(iii)  $m_{kl}m_{ij}-h_{m_{kl}}^{-1}(m_{ij})m_{kl}  \in R(m_{kl}^-)$ for $(i,j)< (k,l)$ 
in lexicographic order

(iv) $n_{kl}m_{ij} = h_{n_{kl}}^{-1}(m_{ij})n_{kl}$ for all $(i,j)$ and $(k,l)$

(v) $n_{kl}n_{ij}-h_{n_{kl}}^{-1}(n_{ij})n_{kl}  \in R(n_{kl}^-)$ for $(i,j)< (k,l)$
in lexicographic order.
 \end{proposition} 
 
 We can use this proposition to show that at each stage in the construction of $T$ we have an Ore extension. In order to do this, we need to utilise the following result concerning Gelfand-Kirillov dimension of extensions of the type considered in the previous result. See \cite{kl} for standard properties of Gelfand-Kirillov dimension. 
 
\begin{lemma}\label{ore-extension-condition-2}
Let $B$ be a $\k$-algebra. Suppose $A$ is a finitely generated subalgebra of $B$ that is an integral domain with finite Gelfand-Kirillov dimension and 
that $x$  is an
element of $B$ such that $B$ is generated by $A$ and $x$ as an algebra. Furthermore, suppose
there exists an automorphism $\sigma$ of $A$ and  finite-dimensional subspace $V$ of $A$ that generates $A$ as an algebra such that $\sigma(V)=V$.  Suppose that 
$xa-\sigma(a)x\in A$, 
for each $a\in A$. Then $\gkdim(B)\leq\gkdim(A)+1$. 

Also, \\
(i) $\delta: A\longrightarrow A$, defined by $\delta(a):=xa-\sigma(a)x$, is a 
$\sigma$-derivation of $A$,  and \\
(ii) if $C:=A[y;\sigma,\delta]$, the natural algebra morphism $\theta \, : \, C \longrightarrow B$ such 
that 
$\theta_{| A} = {\rm id}_A$ and $\theta(y)=x$ is an isomorphism if 
only if $\gkdim(B)=\gkdim(A)+1$. 
\end{lemma} 

\begin{proof} 
Note that \cite[Lemma 2.3]{lr-qsch} guarantees that $\gkdim(B)\leq\gkdim(A)+1$. 
As $C$ is a particular example of such a $B$, we have $\gkdim(C)\leq\gkdim(A)+1$.
 However, it is well-known that $\gkdim(C)\geq\gkdim(A)+1$ (see \cite[p.164]{kl}) and so 
$\gkdim(C)=\gkdim(A)+1$.

It is routine to check that $\delta$ is a $\sigma$-derivation of $A$. 
The map $\theta:C\longrightarrow B$ given by $\theta(f(y)):=f(x)$ is an epimorphism from $C$ 
to $B$. If $\theta$ is not an isomorphism then $\gkdim(B) \le \gkdim(C)-1=\gkdim(A)$, by 
\cite[Proposition 3.15]{kl}, 
while if 
$\theta$ is an isomorphism then $\gkdim(B)=\gkdim(C)=\gkdim(A)+1$, as required. 
\end{proof}

 \begin{corollary}\label{corollary-gkdimT}
  $\gkdim(T) \leq (m+n+1)t -\sum_{i=1}^t\,(a_i+b_i)$
 \end{corollary}
 
 \begin{proof}
 Recall that $T$ is generated over $\k$ by the $t^2$ elements $x_{a_ib_j}$ together with those $t\times t$ quantum minors that are greater than $\gamma$ and differ from $\gamma$ in exactly one entry. 
Such a minor which excludes $a_i$ is given by including a row index which is  bigger than $a_i$ but not equal to 
any of the other $a_j$. There are $m-a_i -(t-i)$ such indices. Summing over $i=1,\dots,t$, 
one obtains $mt - \sum a_i -t(t-1)/2$. There are also $nt-\sum b_j -t(t-1)/2$ such quantum minors that exclude 
a $b_j$, giving a total of $(m+n)t -\sum (a_i+b_i) -t(t-1)$ such quantum minors. Adding in $t^2$ for the 
elements $x_{a_ib_j}$  produces $(m+n+1)t -\sum_{i=1}^t\,(a_i+b_i)$ generators.
At each stage that a new generator is introduced, we have an algebra $A$, a new generator $x$ to generate 
an algebra $B$ containing $A$ and 
an automorphism $\sigma$ with the 
property that $xa-\sigma(a)x\in A$ for elements $a$ in a generating set of $A$ as an algebra. As $\sigma$ is an automorphism, 
this property extends to all elements of $A$ and so the first part of Lemma~\ref{ore-extension-condition-2} is applicable to establish that $\gkdim(B)\leq \gkdim(A)+1$. There are 
$(m+n+1)t -\sum_{i=1}^t\,(a_i+b_i)$ such extensions building up $T$ from the base field $k$ and so the required inequality is obtained. 
\end{proof} 



\section{Generalised quantum determinantal rings are maximal orders}

Let $\gamma = [A|B]=[a_1,\dots,a_t\mid b_1,\dots,b_t]$ be a quantum minor in $\oqmnk$ and set $J_{\gamma}=\oqmnkg$. Then $\overline{\gamma}$ 
is a regular normal element of $J_{\gamma}$, by Corollary~\ref{corollary-minimal-element}; and so we can invert $\overline{\gamma}$ to obtain the localisation $J_{\gamma}[\overline{\gamma}^{-1}]$.
 Our aim is to show that  this localisation is isomorphic to a localisation $T[\gamma^{-1}]$ of the algebra $T$ constructed in the previous section. As a consequence, $J_{\gamma}[\overline{\gamma}^{-1}]$ will be a maximal order. From this we will deduce that 
$J_{\gamma}$ is a maximal order. 

There is a natural homorphism $\theta$ from $T[\gamma^{-1}]$ to 
$J_{\gamma}[\overline{\gamma}^{-1}]$, see below. In order to show that $\theta$ is 
surjective, we need to employ quantum Laplace expansions; while in order to see that $\theta$ is 
injective, we need to use Gelfand-Kirillov dimension calculations. The details are in the next few results.

\begin{lemma} \label{lemma-qdot-commutes}
The quantum minor $\gamma$ is a regular normal element in the algebra $T$. More precisely, 
$\gamma$ $\qdot$-commutes with each of the generators of $T$.
\end{lemma} 

\begin{proof}
A variable $x_{ij}$ is in the generating set for $T$ as an algebra precisely when $i\in A$ and $j\in B$, in which case $x_{ij}$ commutes with $\gamma$. 

Recall that  $m_{ij} = [A\mid B\sqcup\{c_{(n-t+1)-i}\}\backslash\{b_j\}]$, with $b_j<c_{(n-t+1)-i}$. 
A simple application of the quantum Muir's law \cite[Proposition 1.3]{lr-qsch} shows that the commutation relation between $\gamma=[A\mid B]$ and $m_{ij}$ is the same as that between $x_{a_1b_j}$ and $x_{a_1c_{(n-t+1)-i}}$, and this is a $q$-commutation as these two variables are on the same row of a quantum matrix. A similar remark applies to the 
commutation relations between $\gamma$ and the $n_{ij}$ with the roles of rows and columns interchanged. 
\end{proof}

As a consequence of the previous lemma, we can form the localisation $T[\gamma^{-1}]$ of $T$ 
obtained by inverting the powers of $\gamma$ and the canonical morphism $T \longrightarrow T[\gamma^{-1}]$
is injective. Gelfand-Kirillov dimension behaves well 
with respect to this localisation, as we see below. 

\begin{lemma} \label{lemma-gk-values}
(i) $\gkdim(T[\gamma^{-1}]) = \gkdim(T)\leq(m+n+1)t -\sum_{i=1}^t\,(a_i+b_i)$.\\
(ii) $\gkdim(J_{\gamma}[\overline{\gamma}^{-1}])=\gkdim(J_{\gamma})=(m+n+1)t -\sum_{i=1}^t\,(a_i+b_i)$.\\
(iii) $\gkdim(T[\gamma^{-1}]) \leq \gkdim(J_{\gamma}[\overline{\gamma}^{-1}])$.
\end{lemma}

\begin{proof}
(i) Let $V$ be the vector space generated by the generators of $T$ (that is, the $x_{ij}, m_{ij}$ and $n_{ij}$). Then the previous lemma shows that 
$\gamma V=V\gamma$. Set $W:= V+\gamma^{-1}\k$ and note that $W$ generates $T[\gamma^{-1}]$ as an 
algebra and that $W\gamma\subseteq T$. Set $Y:= W\gamma +\k\gamma +\k$, a finite dimensional vector subspace of $T$. It is easy to check that $W^n\gamma^n\subseteq Y^n$. 
It follows that $\dim(W^n)\leq \dim(Y^n)$ and so $\gkdim(T[\gamma^{-1}]) \leq \gkdim(T)$. As it is 
obvious that $\gkdim(T[\gamma^{-1}]) \geq \gkdim(T)$, equality follows. The inequality is already established in Corollary~\ref{corollary-gkdimT}.\\
(ii) For the first equality, a similar proof to that in (i) works, taking the generating subspace $V$ to be generated by the image of $\Pi$ in $J_\gamma$, and using Corollary~\ref{corollary-minimal-element} instead of Lemma~\ref{lemma-qdot-commutes}. For the second equality, see \cite[Remark 4.2(iii)]{lr-qsch}.\\
(iii) This is immediate, from (i) and (ii).
\end{proof}

The inclusion of $T$ in $\oqmnk$ induces a natural homomorphism 
$\theta: T\longrightarrow J_{\gamma}[\overline{\gamma}^{-1}]$ which sends any quantum minor $[I\mid J]$ 
to its image $\overline{[I\mid J]}$ in $J_{\gamma} \subseteq J_{\gamma}[\overline{\gamma}^{-1}]$. In 
particular, $\theta(\gamma)=\overline{\gamma}$, and so we may extend $\theta$ to a homomorphism 
(also denoted by $\theta$) from $T[\gamma^{-1}]$ to $J_{\gamma}[\overline{\gamma}^{-1}]$.

\begin{proposition}\label{proposition-isomorphism} 
The homomorphism
\[
\theta:T[\gamma^{-1}]\longrightarrow J_{\gamma}[\,\overline{\gamma}^{\,-1}]
\]
is an isomorphism. 
\end{proposition} 

\begin{proof} 
The algebra $J_{\gamma}[\,\overline{\gamma}^{\,-1}]$ is generated over $\k$ by 
$\overline{\gamma}^{\,\pm 1}$ together with the image $\overline{x_{rs}}$ in $J_\gamma$ of the $mn$ generators 
$x_{rs}$ of $\oqmnk$. Thus, in order to prove surjectivity, it is enough to see that the $\overline{x_{rs}}$ are all in 
the image of $\theta$. 
This is obvious for the $\overline{x_{rs}}$ which have 
$r\in A$ and $s\in B$.

Let $x_{rs}$ be a generator with $(r,s)\notin A \times B$. 

Suppose first that $s\not\in B$. By \cite[A.5. Corollary (b)(i)]{gl-ijm},
for any $I\subseteq\{1,\dots,m\}$ and $J\subseteq\{1,\dots,n\}$ with $|J|=|I|+1$ we have 
\begin{eqnarray*}
 \sum_{j\in J} (-q)^{|[1,j)\cap J|} x_{rj} [I\mid J\setminus
\{j\}] 
&=&  \left\{
\begin{array}{l}
(-q)^{|[1,r)\cap I|} [I\sqcup \{r\} \mid J] \quad
(r\notin I)\\ 0 
\hspace{24ex}(r\in I) \end{array}
\right. .
\end{eqnarray*}

If we set $I=A$ and $J=B\sqcup\{s\}$  then we obtain the following relation in $\oqmnk$:
\begin{eqnarray*}
\minusqdot x_{rs}[A|B] + \sum_{j=1}^t\, \minusqdot x_{rb_j}[A\mid B\sqcup\{s\}\backslash\{b_j\}]\!\!\!
&=&\!\!\! \left\{
\begin{array}{l}
\minusqdot [A\sqcup\{r\}\mid B\sqcup\{s\}]~
(r\notin A)\\ 0 
\hspace{24ex}(r\in A) \end{array}
\right..
\end{eqnarray*}
To start with, suppose in addition that $r \in A$. Let us look at the image in $J_\gamma$ of the above 
relation. The terms $x_{rb_j}[A\mid B\sqcup\{s\}\backslash\{b_j\}]$ where $s < b_j$ are sent to zero because
they are not greater than or equal to $\gamma$. In addition, the image in $J_\gamma$ of the remaining such 
terms is in the image of $\theta$. It follows from this that, in this case, $\overline{x_{rs}}$ is in the 
image of $\theta$.  
 Of course, by a similar argument, exchanging rows and columns, we get that $\overline{x_{rs}}$ is in the
image of $\theta$ whenever $r \notin A$ and $s\in B$. 

It remains to deal with the case where $r\notin A$ and $s\notin B$. In that case, we have the relation
$\minusqdot x_{rs}[A|B] + \sum_{j=1}^t\, \minusqdot x_{rb_j}[A\mid B\sqcup\{s\}\backslash\{b_j\}] 
= \minusqdot [A\sqcup\{r\}\mid B\sqcup\{s\}]$ in $\oqmnk$. In this relation, the right hand term is not 
greater than or equal to $\gamma$ since it is a $(t+1) \times (t+1)$ minor. So, taking the image of this  latter relation
in $J_\gamma$, we get, by the same argument as above, that $\overline{x_{rs}}$ is in the image of $\theta$ 
since, as we have just proved, $\overline{x_{rb_j}}$ is in the image of $\theta$. 

This finishes the proof that $\theta$ is a surjective map. 

Hence, $\gkdim(T[\gamma^{-1}])\geq 
\gkdim(\theta(T[\gamma^{-1}])) = 
\gkdim(J_{\gamma}[\overline{\gamma}^{-1}])=\gkdim(J_{\gamma})$. 
Together with Lemma~\ref{lemma-gk-values}(iii) this gives $\gkdim(T[\gamma^{-1}])=
\gkdim(J_{\gamma}[\overline{\gamma}^{-1}])$.

Suppose now that $\theta$ is not injective. Then $\ker(\theta)$ is a nonzero 
ideal in the noetherian domain $T[\gamma^{\pm 1}]$.  Hence, 
$\gkdim(\theta(T[\gamma^{\pm 1}])) < \gkdim(T[\gamma^{\pm 1}])$, 
by \cite[Proposition 3.15]{kl}. However, this contradicts the fact that these two dimensions are equal, 
as observed in the previous paragraph. Thus, $\theta$ is injective and so $\theta$ is an isomorphism.
\end{proof} 

\begin{corollary}
$\gkdim(T)= (m+n+1)t -\sum_{i=1}^t\,(a_i+b_i)$.
\end{corollary} 

\begin{proof}
This follows immediately from Lemma~\ref{lemma-gk-values}. 
\end{proof}

\begin{corollary} \label{T-is-ioe}
The algebra $T$ is an iterated Ore extension.
\end{corollary} 

\begin{proof}
 The algebra $T$ is constructed from $\k$ by adding in the $(m+n+1)t -\sum_{i=1}^t\,(a_i+b_i)$
 generators one-by-one. At each stage, the Gelfand-Kirillov dimension can increase by at most one, 
 by Lemma~\ref{ore-extension-condition-2}, and so must increase by exactly one, as 
 $\gkdim(T)= (m+n+1)t -\sum_{i=1}^t\,(a_i+b_i)$. Thus, each stage is an Ore extension, by
 Lemma~\ref{ore-extension-condition-2}. 
\end{proof}

\begin{remark} 
Corollary \ref{T-is-ioe} extends Lemma 6.4 of \cite{bv}, which asserts that in the commutative case
(that is when $q=1$), $\co_1(M_{mn}(K))_\gamma$ is a localisation of a polynomial ring in 
$(m+n+1)t -\sum_{i=1}^t\,(a_i+b_i)$ indeterminates.
\end{remark}

\begin{proposition} 
The QGASL $\oqmnkg$ is an integral domain. 
\end{proposition}

\begin{proof} The algebra $T[\gamma^{-1}]$ is an integral domain as it is a localisation of 
$T$ which is a subalgebra of the domain $\oqmnk$. As a consequence, the isomorphism 
$
T[\gamma^{-1}]\cong \oqmnkg[\,\overline{\gamma}^{\,-1}]
$  of Proposition~\ref{proposition-isomorphism} shows that $\oqmnkg[\,\overline{\gamma}^{\,-1}]$ is an integral domain. As $\overline{\gamma}$ is a regular normal element of $\oqmnkg$ by 
Corollary~\ref{corollary-minimal-element}, the natural map $\oqmnkg\longrightarrow\oqmnkg[\overline{\gamma}^{\,-1}]$ is a monomorphism, and so $\oqmnkg$ is also an integral domain. 
\end{proof}

\begin{remark} 
The previous result applies to all generalised quantum determinantal rings $\oqmnk_\tau$, 
for any $\tau\in\Pi$. In particular, it applies to the {\em upper neighbours} of $\gamma$ which are the elements 
$\tau\in\Pi\backslash\Pi_{\gamma}$ with the property that if $\sigma\in\Pi$ with 
$\gamma<_\st\sigma\leq_\st\tau$ then $\sigma=\tau$. This makes available 
\cite[Proposition 2.2.2]{lr-qsch} which we use in the proof of our main theorem below. 
\end{remark} 

\begin{theorem}
The generalised quantum determinantal ring $\oqmnkg$ is a maximal order.
\end{theorem} 

\begin{proof}
The algebra $T[\gamma^{\pm 1}])$ is a localisation of an iterated Ore extension, and so is 
a maximal order, by \cite[V. Proposition 2.5, IV. Proposition 2.1]{mr}. Thus, 
$\oqmnkg[\overline{\gamma}^{-1}]$ is a maximal order, by the isomorphism established in 
Proposition~\ref{proposition-isomorphism}. 
Hence, \cite[Proposition 2.2.2]{lr-qsch} applies to the quantum graded algebra with a straightening 
law $\oqmnkg$ (whose underlying poset has the single minimal element $\gamma$); so we conclude 
that $\oqmnkg$ is a maximal order. 
\end{proof} 

\begin{remark} 
As pointed out in the introduction, our motivation in the present work is to complete the study, from the point of view of noncommutative algebraic geometry, of generalised quantum determinantal rings. 
 Here is a summary of the results.
 
Let $\gamma\in\Pi$. Then, $\oqmnk_\gamma$ is an integral domain and a maximal order in its division ring 
of fractions, as established in \cite{lr-qsch} and the present work.

Further, $\oqmnk_\gamma$ is AS-Cohen-Macaulay, and it is AS-Gorenstein for any nonzero $q$ in $K$ if and 
only if it is AS-Gorenstein for $q=1$.
All this can be shown following the arguments developed in paragraph 4 of \cite{lr-qasl}
(see in particular Theorems 4.2 and 4.3). Notice in addition that necessary and sufficient conditions for   
$\co_1(M_{mn}(K))$ to be AS-Gorenstein are given in \cite[Theorem 8.14]{bv}.

\end{remark} 






\vskip 1cm

\begin{minipage}{40ex}
{\noindent T. H. Lenagan: \\
Maxwell Institute,\\
School of Mathematics,\\ University of Edinburgh,\\
James Clerk Maxwell Building,\\
The King's Buildings,\\
Peter Guthrie Tait Road,\\
Edinburgh EH9 3FD \\
           UK\\[0.5ex]
email: tom@maths.ed.ac.uk\\}
\end{minipage}
~~~~~
\begin{minipage}{40ex}{\noindent L. Rigal: \\
Universit\'e Sorbonne Paris Nord,\\
LAGA, CNRS, UMR 7539,\\
F-93430, Villetaneuse,\\
 France\\[0.5ex]
email: rigal@math.univ-paris13.fr\\~\\~\\~\\~\\}
\end{minipage}


\end{document}